\documentclass[11pt]{amsart}
\usepackage{amsmath}
\usepackage{amssymb}
\usepackage{amscd}
\usepackage{diagrams}
\def\NZQ{\mathbb}               % the font for N,Z,Q,R,C
\def\NN{{\NZQ N}}
\def\QQ{{\NZQ Q}}
\def\ZZ{{\NZQ Z}}

\newtheorem{Theorem}{Theorem}[section]
\newtheorem{Lemma}[Theorem]{Lemma}
\newtheorem{Corollary}[Theorem]{Corollary}
\newtheorem{Proposition}[Theorem]{Proposition}

\let\epsilon\varepsilon
\let\phi=\varphi
\let\kappa=\varkappa

\begin{document}

\title{Ramification of  local rings along valuations}
\author{Steven Dale Cutkosky and  Pham An Vinh}
\thanks{The first author was partially supported by NSF}

\address{Steven Dale Cutkosky, Department of Mathematics,
University of Missouri, Columbia, MO 65211, USA}
\email{cutkoskys@missouri.edu}

\address{Pham An Vinh, Department of Mathematics,
University of Missouri, Columbia, MO 65211, USA}
\email{vapnnc@mizzou.edu}

\begin{abstract} In this paper we discuss stable forms of extensions of algebraic local rings along a valuation in all dimensions 
over a field $k$ of characteristic zero, and generalize a formula of Ghezzi, H\`a and Kashcheyeva describing the extension of associated graded rings along the valuation for stable extensions of regular algebraic local rings of dimension two to arbitrary ground fields $k$ of characteristic zero. We discuss the failure of this result in positive characteristic.
\end{abstract}

\maketitle

\section{Introduction}

Suppose that $k$ is a field, $K$ is an algebraic function field over $k$ and $\nu$ is a valuation of $K$
(which is trivial on $k$). Let $V_{\nu}$ be the valuation ring of $\nu$, with maximal ideal $m_{\nu}$. Let $\Gamma_{\nu}$ be the value group of $\nu$. Important invariants associated to $\nu$ are its rank (one less than the number of prime ideals in $V_{\nu}$), rational rank ($\dim_{\QQ}\Gamma_{\nu}\otimes_{\ZZ}\QQ$) and dimension ($\dim\nu=\mbox{trdeg}_kV_{\nu}/m_{\nu}$). We have that
$\mbox{rank }\nu\le \mbox{rational rank }\nu$ and by Abhyankar's inequality (\cite{Ab1} and Appendix 2 \cite{ZS}),
\begin{equation}\label{eqN22}
\mbox{rational rank }\nu+\dim \nu\le \dim K
\end{equation}
where $\dim K=\mbox{trdeg}_kK$, and if equality holds, then $\Gamma_{\nu}\cong\ZZ^{rr}$ as an unordered group, where $rr=\mbox{rational rank }\nu$. Such valuations are called Abhyankar valuations.

An algebraic local ring of $K$ is a local ring $R$ which is essentially of finite type over $k$ and whose
quotient field is $K$. $R$ is dominated by $\nu$ if $R\subset V_{\nu}$ and $m_{\nu}\cap R=m_R$ is the maximal ideal of $R$.

A monoidal transform $R\rightarrow R_1$ of $R$ is a local ring $R_1$ of the blowup of a regular prime ideal $P$ of $R$ ($R/P$ is regular).
$R\rightarrow R_1$ is a quadratic transform if $R_1$ is a local ring of the blow up of the maximal ideal of $R$.
$R\rightarrow R_1$ is a monoidal transform along $\nu$ if $V_{\nu}$ dominates $R_1$.

 For each $\gamma \in \Gamma_{\nu}$, let
$$
\mathcal{P}_{\gamma}(R) = \{f \in R \mid \nu(f) \geq \gamma\} \text{ and } \mathcal{P}_{\gamma}^+(R) = \{f \in R \mid \nu(f) > \gamma\}.
$$
We define the associated graded algebra of $\nu$ on $R$ (as in \cite{T1}) as
$$
\text{gr}_{\nu}(R) = \bigoplus_{\gamma \in \Gamma_{\nu}} \mathcal{P}_{\gamma}(R)/\mathcal{P}_{\gamma}^+(R).
$$
If $f \in R$ and $\nu(f) = \gamma$, we define the initial form $\text{in}_{\nu}(f)$ of $f$ in $\text{gr}_{\nu}(R)$ as $f + \mathcal{P}_{\gamma}^+(R)\in \mathcal P_{\gamma}/\mathcal P_{\gamma}^+$. A sequence $\{P_i\}_{i \geq 0}$ in $R$ is called a generating sequence of $\nu$ in $R$ if $\{\text{in}_{\nu}(P_i)\}_{i \geq 0}$ generate $\text{gr}_{\nu}(R)$ as an $R/m_R$-algebra.

The semigroup of $R$ is
$$
S^{R}(\nu)=\{ \nu(f)\mid f\in R\setminus\{0\}\}.
$$
Now suppose that  $K^*$ is an algebraic function field over $k$ such that $K^*$ is finite separable over $K$, and $\nu^*$ is a valuation of $K^*$ which is an extension of $\nu$. Let
$$
n=\mbox{trdeg}_kK^*-\mbox{trdeg}_kV_{\nu^*}/m_{\nu^*}.
$$
Let
$$
e=[\Gamma_{\nu^*}:\Gamma_{\nu}]\mbox{ and }f=[V_{\nu^*}/m_{\nu^*}:V_{\nu}/m_{\nu}]
$$
be the reduced ramification index and relative degree of $\nu^*$ over $\nu$.

Suppose that $R$ and $S$ are algebraic local rings for $K$ and $K^*$ such that $S$ dominates $R$ and $\nu^*$ dominates $S$ (so that
$\nu$ dominates $R$).  

\begin{Lemma} \label{LemmaN1} Suppose that $V_{\nu^*}/m_{\nu^*}=(V_{\nu}/m_{\nu})(S/m_{\nu})$. Then $[S/m_S:R/m_R]=f$ if and only if 
$V_{\nu}/m_{\nu}$ and $S/m_S$ are linearly disjoint in $V_{\nu^*}/m_{\nu^*}$ over $R/m_R$.
\end{Lemma}

\begin{proof} Suppose that $[S/m_S:R/m_R]=f$. Let $h_1,\ldots,h_s\in S/m_S$ be linearly independent over $R/m_R$. Extend this set to a basis $h_1,\ldots, h_f$ of $S/m_S$ over $R/m_R$. Then $h_1,\ldots, h_f$ span $V_{\nu^*}/m_{\nu^*}$ over $V_{\nu}/m_{\nu}$, so they are
linearly independent over $V_{\nu}/m_{\nu}$. 

Now suppose that $V_{\nu}/m_{\nu}$ and $S/m_S$ are linearly disjoint over $R/m_R$. There exist $\alpha_1,\ldots,\alpha_f\in S/m_S$ which are a basis of $V_{\nu^*}/m_{\nu^*}$ over $V_{\nu}/m_{\nu}$. Then $\alpha_1,\ldots,\alpha_f$ are linearly independent over $R/m_R$, so $[S/m_S:R/m_R]\ge f$. However, a basis of $S/m_S$ over $R/m_R$ is linearly independent over $V_{\nu}/m_{\nu}$, so 
$[S/m_S:R/m_R]=f$. 
\end{proof}

We will say that $R\rightarrow S$ is {\it monomial} if $R$ and $S$ are $n$-dimensional regular local rings and there exist regular parameters
$x_1,\ldots,x_n$ in $R$, $y_1,\ldots,y_n$ in $S$, an $n\times n$ matrix $A=(a_{ij})$ of natural numbers with $\mbox{Det}(A)\ne 0$ and units $\delta_i\in S$ such that
\begin{equation}\label{eqN2}
x_i=\delta_i\prod_{j=1}^ny_j^{a_{ij}}
\end{equation}
for $1\le i\le n$. In Theorem 1.1 \cite{C} it is proven that when the ground field $k$ has characteristic zero, there exists a commutative diagram
\begin{equation}\label{eqN1}
\begin{array}{ccc}
R_0&\rightarrow&S_0\\
\uparrow&&\uparrow\\
R&\rightarrow&S
\end{array}
\end{equation}
such that the vertical arrows are products of monoidal transforms along $\nu^*$ and $R_0\rightarrow S_0$ is monomial.
It is shown in Theorem 5.1 \cite{C} and Theorem 4.8 \cite{CP} that the matrix $A_0$ describing $R_0\rightarrow S_0$
(with respect to regular parameters $x_1(0),\ldots,x_n(0)$ in $R_0$ and $y_1(0),\ldots,y_n(0)$ in $S_0$) can be required to take a very special block form, which reflects the rank and rational rank of $\nu^*$. We will say that $R_0\rightarrow S_0$ is {\it strongly monomial} if it is monomial and the matrix $A_0$ has this special block form.

In Theorem 6.1 \cite{CP} it is shown (assuming that $k$ has characteristic zero) that we can always find a diagram (\ref{eqN1}) such that the following conditions hold:
\begin{enumerate}
\item[1)] $R_0\rightarrow S_0$ is strongly monomial.
\item[2)]  If
$$
\begin{array}{ccc}
R_1&\rightarrow&S_1\\
\uparrow&&\uparrow\\
R_0&\rightarrow&S_0
\end{array}
$$
is such that $R_1\rightarrow S_1$ is strongly monomial with respect to regular parameters $x_1(1),\ldots,x_n(1)$ in $R_1$ and
$y_1(1),\ldots,y_n(1)$ in $S_1$, and the vertical arrows are products of monoidal transforms, then 
\begin{enumerate}
\item[2a)]
The natural group homomorphism
$$
\ZZ^n/A_1^t\ZZ^n\rightarrow \Gamma_{\nu^*}/\Gamma_{\nu}
$$
defined by
$$
(b_1,\ldots,b_n)\mapsto [b_1\nu^*(y_1(1))+\cdots+b_n\nu^*(y_n(1))]
$$
is an isomorphism (where $A_1$ is the matrix of exponents of $R_1\rightarrow S_1$ with respect to our given systems of parameters).
\item[2b)] $V_{\nu^*}/m_{\nu^*}$ is the join $V_{\nu^*}/m_{\nu^*}=(V_{\nu}/m_{\nu})(S_1/m_{S_1})$.
\item[2c)] $V_{\nu}/m_{\nu}$ and $S_1/m_{S_1}$ are linearly disjoint over $R_1/m_{R_1}$ in $V_{\nu^*}/m_{\nu^*}$.
\end{enumerate}
\end{enumerate}

Theorem 6.1 \cite{CP} and Lemma \ref{LemmaN1} implies that, given $R\rightarrow S$, there exists a monomial extension $R_0\rightarrow S_0$ as in (\ref{eqN1}) satisfying 1) and 2) above. In Theorem 6.3 \cite{CP} it is shown that the extension $V\rightarrow V^*$ can naturally be understood as a direct limit of $R_0\rightarrow S_0$ as above.

We will say that $R_0\rightarrow S_0$ is {\it stable} if the conclusions 1) and 2) above hold.

If $R\rightarrow S$ is stable, we have that
$$
e=\mbox{Det}(A)\mbox{ and }f=[S/m_S:R/m_R].
$$
where $e$ is  the reduced ramification index and $f$ is the relative degree of $\nu^*$ over $\nu$.

The simplest valuations are the Abhyankar valuations (defined at the beginning of this section). In this case,   we easily obtain
a very strong statement comparing the associated graded rings of the valuations. We have the following proposition.

\begin{Proposition}\label{Prop1} Suppose that $k$ has characteristic zero, $\nu$ is an Abhyankar valuation and $R\rightarrow S$ is stable. Then we have a natural isomorphism of graded rings
$$
\mbox{gr}_{\nu^*}(S)\cong \left(\mbox{gr}_{\nu}(R)\otimes_{R/m_R}S/m_S\right)[\overline y_1,\ldots,\overline y_n]
$$
where $\overline y_1,\ldots,\overline y_n$ are the initial forms of $y_1,\ldots,y_n$, with the only relations being
$$
[x_i]=[\delta_i]\overline y_1^{a_{i1}}\cdots \overline y_n^{a_{in}}\,\,\,1\le i\le n
$$
obtained from (\ref{eqN2}) ($[\delta_i]$ is the class of $\delta_i$ in $S/m_S$). The degree of the extension of quotient fields of
$$
\mbox{gr}_{\nu}(R)\rightarrow \mbox{gr}_{\nu^*}(S)
$$
is $ef$.
\end{Proposition}

\begin{proof}
Since $R\rightarrow S$ is stable, and $\nu^*$ and $\nu$ are Abhyankar valuations, we have that
$\nu^*(y_1),\ldots,\nu^*(y_n)$ is a $\ZZ$-basis of $\Gamma_{\nu^*}$ and $\nu(x_1),\ldots,\nu(x_n)$ is a $\ZZ$-basis of $\Gamma_{\nu}$.

By Hensel's lemma, $\hat R\cong k'[[x_1,\ldots,x_n]]$ where $k'\cong R/m_R$ is a coefficient field of $\hat R$.
Since $\nu(x_1),\ldots,\nu(x_n)$ are rationally independent,
$\nu$ has a unique extension to a valuation $\hat\nu$ of the quotient field of $\hat R$, defined by
$$
\hat\nu(f)=\min\{i_1\nu(x_1)+\cdots+i_n\nu(x_n)\mid a_{i_1,\ldots,i_n}\ne 0\}
$$
if $f=\sum a_{i_1,\ldots,i_n}x_1^{i_1}\cdots x_n^{i_n}\in k'[[x_1,\ldots,x_n]]$ (with $a_{i_1,\ldots,i_n}\in k'$).
Since distinct monomials have distinct values, we have an isomorphism of  residue fields  $V_{\nu}/m_{\nu}\cong R/m_R$.

Hence ${\rm gr}_{\nu}(R)\cong R/m_R[\overline x_1,\ldots,\overline x_n]$, is a polynomial ring, where
$\overline x_i$ is the class of $x_i$, with the grading $\deg \overline x_i=\nu(x_i)$. Further ${\rm gr}_{\nu^*}(S)\cong S/m_S[\overline y_1,\ldots,\overline y_n]$, is a polynomial ring, where
$\overline y_i$ is the class of $y_i$. The proposition follows.
\end{proof}

If $\nu$ is an Abhyankar valuation, and $R\rightarrow S$ is quasi-finite, we   have that $S^S(\nu^*)$ is  finitely generated as a module over the semigroup $S^R(\nu)$ by Proposition \ref{Prop1}.

It is natural to ask if an analog of Proposition \ref{Prop1} holds for more general valuations.  We have the essential difference that   the valuation groups $\Gamma_{\nu}$ are not finitely generated in general. There even exist examples where $R\rightarrow S$ is quasi-finite but $S^S(\nu^*)$ is not a finitely generated module over $S^R(\nu)$. In Theorem 9.4 \cite{CV} an example is given of a finite monomial
extension of two dimensional regular algebraic local rings (over any ground field) such that $S^S(\nu^*)$ is not a finitely generated module over $S^R(\nu)$. This example is necessarily not stable.
Some other examples are given in \cite{CT} showing bad behavior of  $S^S(\nu^*)$ over $S^R(\nu)$.

However, the conclusions of Proposition \ref{Prop1} always hold for stable mappings $R\rightarrow S$ when $R$ and $S$ have dimension two
($n=2$).
By Abhyankar's inequality, when $n=2$, $\nu$ is an Abhyankar valuation unless $\nu$ is rational (the value group is order isomorphic to a subgroup of the rational numbers). We have the following theorem, which generalizes Proposition \ref{Prop1} to this case.
This surprising theorem was proven when $k$ is algebraically closed of characteristic zero and $\dim K=2$ by Ghezzi, H\`a and Kashcheyeva in \cite{GHK}. If $n=2$, $\nu$ is rational and $R\rightarrow S$ is stable, then $R$ has regular parameters $u,v$, $S$ has regular parameters $x,y$ and there exist a unit $\gamma$ in $S$ such that
\begin{equation}\label{eqN23}
u=\gamma x^e, v=y,
\end{equation}
where $e=|\Gamma_{\nu^*}/\Gamma_{\nu}|$ is the reduced ramification index.

\begin{Theorem}\label{Theorem2} Suppose that $k$ is a field of characteristic zero, $\nu^*$ is a rational 0-dimensional valuation, $n=2$  and $R\rightarrow S$ is stable.
Then
$$
{\rm gr}_{\nu^*}(S)\cong \left({\rm gr}_{\nu}(R)\otimes_{R/m_R}S/m_S\right)[Z]/(Z^e-[\gamma_0]^{-1}[u]),
$$
and the degree of the extension  of quotient fields of  ${\mbox gr}_{\nu}(R)\rightarrow \mbox{gr}_{\nu^*}(S)$ is $ef$.
\end{Theorem}

The remaining sections of this paper are devoted to the proof of this theorem. Our proof requires the construction of generating sequences
for valuations in arbitrary regular local rings of dimension two in \cite{CV}. Theorem \ref{Theorem2} is proven in Section \ref{Sec4}, as a consequence of Proposition \ref{Prop30}, which shows that a generating sequence in $R$ is almost a generating sequence in $S$ if $R\rightarrow S$ is stable.

In contrast to the fact that finite generation may not hold even for a monomial mapping, when $\nu^*$ is a rational 0-dimensional valuation with $n=2$ (Example 9.4 \cite{CV}), we have finite generation if $R\rightarrow S$ is stable.

\begin{Corollary} Suppose that $k$ is a field of characteristic zero, $\nu^*$ is a rational 0-dimensional valuation, $n=2$  and $R\rightarrow S$ is stable.
Then the semigroup $S^{S}(\nu^*)$ is a finitely generated $S^{R}(\nu)$-module.
\end{Corollary}

An interesting question is if an analogue of the conclusions of Proposition \ref{Prop1} holds in general for any $n$ and arbitrary valuations for stable mappings over fields $k$ of characteristic zero. It would be remarkable if this were true.

With some small modification in the definition of strongly monomial (in (3)), strong monomialization holds for Ahhyankar valuations in positive characteristic, as follows from \cite{KK}, (a strong form of  local uniformization is proven for Abhyankar valuations by Knaf and Kuhlmann), and thus Proposition \ref{Prop1} holds in positive characteristic. A description of ${\rm gr}_{\nu}(R)$ for $\nu$ an Abhyankar valuation dominating  a (singular)  local ring $R$, over an algebraically closed field of arbitrary characteristic, and a proof of local uniformization for Abhyankar valuations derived from this construction, has been recently given by Teissier in \cite{T}.

Over fields of  positive characteristic, it is shown in Section 7.11 of \cite{CP} that the strong monomialization theorem is not  true, even when $n=2$, $k$ is algebraically closed and $\nu$ is rational and zero dimensional. It is not known if monomialization holds, although it seems unlikely.
 Stable forms are given in \cite{CP}  for mappings in dimension two over an algebraically closed field of positive characteristic
which are much more complicated than in the characteristic zero case. The fundamental obstruction to obtaining strong monomialization is the defect. It is shown in \cite{CP} that strong monomialization holds in dimension two over algebraically closed fields $k$ for extensions of valuations for which there is no defect. In \cite{GH}, Ghezzi and Kashcheyeva prove Theorem \ref{Theorem2} when $k$ is algebraically closed of positive characteristic, $\dim K=2$  and the
extension has no defect.

In the example of Section 7.11 of \cite{CP}, the stable forms $R_i\rightarrow S_i$ satisfy
\begin{equation}\label{eqN24}
{\rm gr}_{\nu}(R_i)\rightarrow {\rm gr}_{\nu^*}(S_i)
\end{equation}
is integral but not finite, in contrast to the case of Proposition \ref{Prop1} and Theorem \ref{Theorem2}. In fact,
${\rm gr}_{\nu}(R_i)={\rm gr}_{\nu^*}(S_i)^p$. Further, $S^{S_i}(\nu^*)$ is not a finitely generated $S^{R_i}(\nu)$-module for any $i$.
In this example, the degree of the extension of quotient fields of (\ref{eqN24}) is $efp^{\delta(\nu^*/\nu)}=p^2$, where $\delta(\nu^*/\nu)=2$ is the defect of $\nu^*$ over $\nu$. The defect is always zero in characteristic zero, and for Abhyankar valuations.

\section{A modification of the algorithm of \cite{CV} to construct a generating sequence}\label{Genseq}

Suppose that $k$ is a field of characteristic zero and $K$ is a two dimensional algebraic function field over $k$.
Suppose that $\nu$ is a rational 0-dimensional valuation of $K$ (the value group is isomorphic as an ordered group to a subgroup of $\QQ$ and $\mbox{trdeg}_kV_{\nu}/m_{\nu}=0$). 
Suppose that $R$ is a  regular algebraic local ring of $K$  such that $\nu$ dominates $R$.

Let
$$
R\rightarrow T_1\rightarrow T_2\rightarrow \cdots
$$
be the sequence of quadratic transforms of $R$ along $\nu$, so that $V_{\nu}=\cup T_i$. Suppose that $x,y$ are regular parameters in $R$. There exists a smallest value $i$ such that the divisor of $xy$ in $\mbox{spec}(T_i)$ has only one component.
\begin{equation}\label{eq1}
{\rm Define}\,\, R_1=T_i.
\end{equation}

We consider the algorithm of Theorem 4.2 \cite{CV} to construct a generating sequence in $R$ with Remark 4.3 \cite{CV} and the following observation: We can replace $U_i$ with a unit $\tau_i\in R$ times $U_i$ in the algorithm. The algorithm (which we will call the modified algorithm to construct a generating sequence) iterates in the following way. Suppose that for $i\ge 0$ we have constructed the first $i+1$ terms
$$
P_0=x, P_1=y, P_2,\ldots, P_i
$$
of a generating sequence by the (modified) algorithm. To produce the next term $P_{i+1}$, the algorithm proceeds as follows.
First we compute
$$
\overline n_i=[G(\nu(P_0),\ldots,\nu(P_i)):G(\nu(P_0),\ldots, \nu(P_{i-1}))].
$$
This allows us to find a suitable element
\begin{equation}\label{eqN8}
U_i=P_0^{\omega_0(i)}P_1^{\omega_1(i)}\cdots P_{i-1}^{\omega_{i-1}(i)}\tau_i
\end{equation}
with $\tau_i\in R$ an arbitary  unit, such that $\nu(P_i^{\overline n_i})=\nu(U_i)$. Let
\begin{equation}\label{E2}
\alpha_i=\left[\frac{P_i^{\overline n_i}}{U_i}\right]\in V_{\nu}/m_{\nu},
\end{equation}
and
$$
f_i(z)=z^{d_i}+b_{i,d_i-1}z^{d_i-1}+\cdots +b_{i,0}
$$
 be the minimal polynomial of $\alpha_i$ over $R/m_R(\alpha_1,\ldots,\alpha_{i-1})$. Then the algorithm produces an element
 $P_{i+1}\in R$ of the form
 \begin{equation}\label{eq6}
 P_{i+1}=P_i^{\overline n_id_i}+\sum_{t=0}^{d_i-1}\left(\sum_{s=1}^{\lambda_t}a_{s,t}P_0^{j_0(s,t)}\cdots P_{i-1}^{j_{i-1}(s,t)}\right)P_i^{t\overline n_i}
\end{equation}
where $a_{s,t}\in R$ are units, $j_0(s,t),\ldots, j_{i-1}(s,t)\in \NN$ with $0\le j_k(s,t)<n_k$ for $k\ge 1$ and $0\le t<d_i$
such that
$$
\nu(P_0^{j_0(s,t)}\cdots P_{i-1}^{j_{i-1}(s,t)}P_i^{t\overline n_i})=\overline n_id_i\nu(P_i)
$$
for all $s,t$, and
\begin{equation}\label{eq7}
b_{i,t}=\left[\sum_{s=1}^{\lambda_t}a_{s,t}   \frac{P_0^{j_0(s,t)}\cdots P_{i-1}^{j_{i-1}(s,t)}}{U_i^{d_i-t}}\right] \in V_{\nu}/m_{\nu}.
\end{equation}
Then $P_0,P_1,\ldots, P_i,P_{i+1}$ are the first $i+2$ terms of a generating sequence for $\nu$ in $R$.

The observation of Remark 4.3 \cite{CV} is that any  choice of (\ref{eq6}) such that (\ref{eq7}) holds gives an extension $P_{i+1}$
to the next term in a generating sequence, satisfying the conclusions of Theorem 4.2 \cite{CV}.

We will consider the (modified) algorithm of Theorem 4.2 \cite{CV} in various rings $R$ with given regular parameters $x,y$. We will denote
$$
P_i(R)=P_i \mbox{ so } P_0(R)=x, P_1(R)=y,
$$
$$
\overline n_i(R)=\overline n_i, U_i(R)=U_i, \alpha_i(R)=\alpha_i, f_i^R(z)=f_i(z), d_i(R)=d_i, n_i(R)=n_i=d_i\overline n_i.
$$
These calculations not only depend on $R$, but on the previous terms $P_0,P_1,\ldots, P_{i-1}$ constructed in the algorithm.
\vskip .2truein
We will also consider the algorithm of Theorem 7.1 \cite{CV} in different rings $R$, with given regular parameters $x,y$, and a generating
sequence
$$
x=P_0, y=P_1, P_2,\ldots, P_i,\ldots
$$
 constructed by the (modified) algorithm 4.2 of \cite{CV}. This algorithm considers the birational  extension ring $R_1$ of $R$
 defined by (\ref{eq1}).

 The positive integers $\overline n_1$ and $\omega_0(1)$ of Theorem 4.2 \cite{CV} are defined by the conditions that
$\overline n_1\nu(y)=\omega_0(1)\nu(x)$ and $\mbox{gcd}(\overline n_1,\omega_0(1))=1$.
Choose $a,b\in \NN$ so that $\overline n_1b-\omega_0(1)a=1$. Define
\begin{equation}\label{E4}
x_1=\frac{x^b}{y^a},\,\,\, y_1=\frac{y^{\overline n_1}}{x^{\omega_0(1)}}.
\end{equation}

Let
\begin{equation}\label{eqN9}
\sigma=[y_1]\in V_{\nu}/m_{\nu},
\end{equation}
 which is nonzero. Then (as is shown in Theorem 7.1 \cite{CV}) $R_1/m_{R_1}=R/m_R[\sigma]$.
 Theorem 7.1 \cite{CV} shows that
 \begin{equation}\label{eqN10}
 Q_0=x_1, Q_1=\frac{P_2}{x_1^{\omega_0(1)n_1}}
 \end{equation}
 are regular parameters in $R_1$, and taking
 \begin{equation}\label{eqN11}
 Q_i=\frac{P_{i+1}}{Q_0^{\omega_0(1)n_1\cdots n_i}}
 \end{equation}
 for $1\le i$, the $Q_i$ are a generating sequence for $\nu$ in $R_1$ produced by the algorithm of Theorem 4.2 \cite{CV}
 (as interpreted by Remark 4.3 \cite{CV}).

 We will consider the algorithm of Theorem 7.1 \cite{CV} in different rings $R$, and will denote
 $$
 Q_i(R_1)=Q_i, \hat{\beta_i}(R_1)=\hat{\beta_i}, V_i(R_1)=V_i, \hat{\alpha_i}(R_1)=\hat{\alpha}_i
 $$
 in the notation of the proof of Theorem 7.1 \cite{CV}.

 We have that the $V_i(R_1)$ constructed in the the proof of Theorem 7.1 \cite{CV} for $R\rightarrow R_1$ are actually the $U_i(R_1)$
 as constructed by the algorithm of Theorem 4.2 \cite{CV}.

 Let $L_0\cong R/m_R$ be a coefficient field of $\hat R$, so that $\hat R= L_0[[x,y]]$.
 $$
 R_1 = R[x_1,y_1]_{m_{\nu}\cap R[x_1,y_1]}.
 $$
  $\nu(x_1)>0$ and $\nu(y_1)=0$. We have that
 $$
 R_1/m_{R_1}\cong L_0[\sigma],
 $$
 where $\sigma$ is the class of $y_1$ in $R_1/m_{R_1}$.
 Let $L_1\cong L_0(\sigma)$ be a coefficient field of $\hat R_1$ containing $L_0$ (this is possible since $k$ has characteristic zero, by Hensel's Lemma).
 Let
 \begin{equation}\label{eq8}
 y_1^*=y_1-\sigma\in m_{\hat{R_1}}.
 \end{equation}
 Thus $x_1, y_1^*$ are regular parameters in $\hat{R_1}$ and hence
 $$
 \hat{R_1}=L_1[[x_1,y_1^*]].
 $$
We have an expression
$$
 x=x_1^{\overline n_1}(y_1^*+\sigma)^a, y=x_1^{\omega_0(1)}(y_1^*+\sigma)^b
 $$
 in $\hat R_1$.

\section{Monomial forms under sequences of  quadratic transforms}\label{SecStable}

Suppose that $k$ is a field of characteristic zero, and $K\rightarrow K^*$ is an extension of two dimensional algebraic function fields over $k$.
Suppose that $\nu^*$ is a rational 0-dimensional valuation of $K^*$ which restricts to $\nu$.
Suppose that $R$ and $S$ are regular algebraic local rings of $K$ and $K^*$ respectively such that $\nu^*$ dominates $S$ and $S$ dominates $R$.

By Theorem 5.1 \cite{C} and Theorem 4.8 \cite{CP} (summarized  after (\ref{eqN1})), there exists a  sequence of quadratic transforms along $\nu^*$
$$
\begin{array}{ccc}
R'&\rightarrow &S'\\
\uparrow&&\uparrow\\
R&\rightarrow &S
\end{array}
$$
such that $R'\rightarrow S'$ is strongly monomial. For this type of valuation, this means that $R'$ has a regular system of parameters $u,v$ and $S'$ has a regular system of parameters $x,y$ giving an expression
\begin{equation}\label{E1}
u=\gamma_0 x^t, v=y
\end{equation}
where $\gamma_0$ is a unit in $S'$.  For the rest of this section, we will  assume that $R\rightarrow S$ is strongly monomial (so $R$, $S$ have regular parameters satisfying (\ref{E1})), but we do not assume that $R\rightarrow S$ is stable.

\begin{Lemma}\label{Lemma2} Suppose that $R$ has regular parameters $u,v$ and $S$ has regular parameters $x,y$ giving an expression
$$
u=\gamma_0 x^t , v=y
$$
where $\gamma_0$ is a unit in $S$. Let $R\rightarrow R_1$ be the sequence of quadratic transforms along $\nu$ defined by
(\ref{eq1}) and  Let $S\rightarrow S_1$ be the sequence of quadratic transforms along $\nu^*$ defined by
(\ref{eq1}). Then $R_1$ has regular parameters $u_1,\tilde v_1$ and $S_1$ has regular parameters $x_1,\tilde y_1$ such that
$$
u_1=\gamma_1x_1^{t_1}, \tilde v_1=\tilde y_1
$$
where $\gamma_1$ is a unit in $S_1$.
\end{Lemma}

\begin{proof} We use the notation of the previous section. We have that
$$
R_1=R[u_1,v_1]_{m_{\nu}\cap R[u_1,v_1]},
$$
$$
u=u_1^{\overline n_1(R)}v_1^{a(R)}, v=u_1^{\omega_0(1)(R)}v_1^{b(R)}
$$
with
$$
\overline n_1(R)b(R)-\omega_0(1)(R)a(R)=1.
$$
We have that
$$
v_1=\frac{v^{\overline n_1(R)}}{u^{\omega_0(1)(R)}}
$$
so that $\nu(v_1)=0$ and $[v_1]=\sigma(R)$ in $V_{\nu}/m_{\nu}$. We have
$$
u_1=\frac{u^{b(R)}}{v^{a(R)}}.
$$
Further,
$$
S_1=S[x_1,y_1]_{m_{\nu^*}\cap S[x_1,y_1]}
$$
where
$$
x=x_1^{\overline n_1(S)}y_1^{a(S)}, y=x_1^{\omega_0(1)(S)}y_1^{b(S)}
$$
with $\overline n_1(S)b(S)-\omega_0(1)(S)a(S)=1$. We have that
$$
y_1=\frac{y^{\overline n_1(S)}}{x^{\omega_0(1)(S)}}
$$
so that $\nu^*(y_1)=0$, and $[y_1]=\sigma(S)$ in $V_{\nu^*}/m_{\nu^*}$. We have
\begin{equation}\label{eq111}
x_1=\frac{x^{b(S)}}{y^{a(S)}}.
\end{equation}
Substitute
$$
\begin{array}{lll}
u_1&=& u^{b(R)}v^{-a(R)}=\gamma_0^{b(R)}x^{tb(R)}y^{-a(R)}\\
&=& \gamma_0^{b(R)}(x_1^{\overline n_1(S)}y_1^{a(S)})^{tb(R)}(x_1^{\omega_0(1)(S)}y_1^{b(S)})^{-a(R)}\\
&=& \gamma_0^{b(R)}x_1^{\overline n_1(S)tb(R)-\omega_0(1)(S)a(R)}y_1^{a(S)tb(R)-a(R)b(S)}.
\end{array}
$$
Set $t_1=\overline n_1(S)tb(R)-\omega_0(1)(S)a(R)$. Since $\nu^*(u_1)>0$, $\nu^*(x_1)>0$ and $\nu^*(\gamma_0)=\nu^*(y_1)=0$,
we have $t_1>0$.

$$
\begin{array}{lll}
v_1&=& u^{-\omega_0(1)(R)}v^{\overline n_1(R)}=(\gamma_0x^t)^{-\omega_0(1)(R)}y^{\overline n_1(R)}\\
&=& \gamma_0^{-\omega_0(1)(R)}(x_1^{\overline n_1(S)}y_1^{a(S)})^{-t\omega_0(1)(R)}(x_1^{\omega_0(1)(S)}y_1^{b(S)})^{\overline n_1(R)}\\
&=& \gamma_0^{-\omega_0(1)(R)}x_1^{\omega_0(1)(S)\overline n_1(R)-t\omega_0(1)(R)\overline n_1(S)}y_1^{b(S)\overline n_1(R)-a(S)t\omega_0(1)(R)}.
\end{array}
$$
$\nu^*(v_1)=\nu^*(y_1)=\nu^*(\gamma_0)=0$ and $\nu^*(x_1)>0$ implies
$$
\omega_0(1)(S)\overline n_1(R)-t\omega_0(1)(R)\overline n_1(S)=0.
$$
 Since $\overline n_1(S)b(S)-\omega_0(1)(S)a(S)\ne 0$, 
we  have that

$$
\left(\begin{array}{ll}
\overline n_1(S)& \omega_0(1)(S)\\
a(S)&b(S)\end{array}\right)
\left(\begin{array}{l} -t\omega_0(1)(R)\\ \overline n_1(R)\end{array}\right)\ne
\left(\begin{array}{l} 0\\ 0\end{array}\right).
$$
Thus
\begin{equation}\label{eqN7}
m:= b(S)\overline n_1(R)-a(S)t\omega_0(1)(R)\ne 0.
\end{equation}
We  have that $u_1,v_1\in S_1$, so that
$$
R_1=R[u_1,v_1]_{m_{\nu}\cap R[u_1,v_1]}\subset S_1.
$$
We have a commutative diagram
$$
\begin{array}{ccc}
\hat{R_1}=L_1[[u_1,v_1^*]]&\rightarrow &\hat{S_1}=M_1[[x_1,y_1^*]]\\
\uparrow&&\uparrow\\
\hat{R}=L[[u,v]]&\rightarrow & \hat{S}=M[[x,y]]
\end{array}
$$
where $L,M,L_1,M_1$ are coefficient fields of $\hat{R}, \hat{S}, \hat{R_1}, \hat{S_1}$ such that there are inclusions
$$
\begin{array}{ccc}
L_1&\rightarrow& M_1\\
\uparrow&&\uparrow\\
L&\rightarrow &M
\end{array}
$$
This is possible (by Hensel's Lemma) since $R,S, R_1, S_1$ have equicharacteristic zero.

$y_1^*=y_1-\sigma(S)$, $v_1^*=v_1-\sigma(R)$ are constructed as in (\ref{eq8}). We compute in $M_1[[x_1,y_1^*]]$,
$$
y_1^m=(y_1^*+\sigma(S))^m=\sigma(S)^m+m\sigma(S)^{m-1}y_1^*+\frac{m(m-1)}{2!}\sigma(S)^{m-2}(y_1^*)^2+\cdots
$$
$\gamma_0^{-\omega_0(1)(R)}=\beta+x_1\Omega$ with $0\ne \beta\in M_1$ and $\Omega\in \hat{S_1}$.

In $\hat{S_1}$ we have an expression
$$
\begin{array}{lll}
v_1&=& (\beta+x_1\Omega)(\sigma(S)^m+m\sigma(S)^{m-1}y_1^*+(y_1^*)^2\Lambda)\\
&=&  \beta\sigma(S)^m+\beta m\sigma(S)^{m-1}y_1^*+x_1\Omega'+(y_1^*)^2\Lambda'
\end{array}
$$
for some $\Lambda\in \hat{S_1}$, $\Omega', \Lambda'\in\hat{S_1}$. Thus $x_1, v_1-\beta\sigma(S)^m$ are regular parameters in $\hat{S_1}$, (and $\beta\sigma(S)^m=\sigma(R)$). Hence if $u_1,v_1'$ are regular parameters in $R_1$, then $x_1, y_1'=v_1'$ are regular parameters in $S_1$, and we have an expression:
$$
\begin{array}{lll}
u_1&=& \gamma_1x_1^{t_1}\\
v_1'&=&y_1'
\end{array}
$$
with $\gamma_1$ a unit in $S_1$.
\end{proof}

By iteration of Lemma \ref{Lemma2} and (\ref{eq1}), we obtain an infinite sequence

$$
\begin{array}{lll}
\vdots &&\vdots\\
\uparrow&&\uparrow\\
R_2&\rightarrow &S_2\\
\uparrow&&\uparrow\\
R_1&\rightarrow &S_1\\
\uparrow&&\uparrow\\
R&\rightarrow &S
\end{array}
$$
where each $R_i$ has regular parameters $u_i,\tilde v_i$ and each $S_i$ has regular parameters $x_i,\tilde y_i$
such that
$$
u_i=\gamma_ix_i^{t_i}, \tilde v_i=\tilde y_i
$$
where $\gamma_i$ is a unit in $S_i$.

Let $e=|\Gamma_{\nu^*}/\Gamma_{\nu}|$ and $f=[V_{\nu^*}/m_{\nu^*}:V_{\nu}/m_{\nu}]$.
If $R\rightarrow S$ is  stable, then
\begin{equation}\label{eq3}
t_i=e\mbox{ and }[S_i/m_{S_i}:R_i/m_{R_i}]=f
\end{equation}
for $i\ge 0$.

\section{Construction of a generating sequence in $S$ from that of $R$}\label{Sec4}

In this section, we continue to have the assumptions of Section \ref{SecStable}.
We further assume that $R\rightarrow S$ is stable.
Let
$$
P_0(R)=u, P_1(R)=v, P_2(R),\ldots
$$
be a generating sequence in $R$, constructed by the algorithm of Theorem 4.2 \cite{CV}.

Let $P_0(S)=x$, $P_1(S)=y$.

Then we have that the $t$ and $t_1$ in Lemma \ref{Lemma2} satisfy
\begin{equation}\label{eq11}
t=|\Gamma_{\nu^*}/\Gamma_{\nu}|=t_1,
\end{equation}
and
\begin{equation}\label{eq9}
[S/m_S:R/m_R]=[V_{\nu^*}/m_{\nu^*}:V_{\nu}/m_{\nu}]=[S_1/m_{S_1}:R_1/m_{R_1}].
\end{equation}
By the calculations in the previous section, we have that
\begin{equation}\label{eq10}
\left(\begin{array}{ll} b(R)& -a(R)\\ -\omega_0(1)(R)&\overline n_1(R)\end{array}\right)
\left(\begin{array}{ll}t&0\\0&1\end{array}\right)
\left(\begin{array}{ll} \overline n_1(S)& a(S)\\ \omega_0(1)(S)&b(S)\end{array}\right)
=\left(\begin{array}{ll} t_1 & *\\ 0&m\end{array}\right).
\end{equation}
Taking determinants and using the fact that $t_1=t$ gives $t=tm$ so that $m=1$.

Multiplying (\ref{eq10}) by
$$
\left(\begin{array}{ll} \overline n_1(R)& a(R)\\ \omega_0(1)(R) & b(R)\end{array}\right),
$$
we obtain
\begin{equation}\label{eq12}
\overline n_1(S)=\overline n_1(R), \omega_0(1)(S)=t\omega_0(1)(R).
\end{equation}
Since
$$
P_1(S)^{\overline n_1(S)}=P_1(R)^{\overline n_1(R)},
$$
 we can take $U_1(S)$ to be $U_1(R)=u^{\omega_0(1)(R)}$,
 so
 $$
 U_1(S)=u^{\omega_0(1)(R)}=\gamma_0^{\omega_0(1)(R)}x^{t\omega_0(1)(R)}=\gamma_0^{\omega_0(1)(R)}x^{\omega_0(1)(S)}.
 $$
 That is, we take $\tau_1=\gamma_0^{\omega_0(1)(R)}$ in (\ref{eqN8}).
  Thus
 $$
 \alpha_1(S)=\left[\frac{P_1(S)^{\overline n_1(S)}}{U_1(S)}\right]
 =\left[\frac{P_1(R)^{\overline n_1(R)}}{U_1(R)}\right] = \left[\frac{v^{\overline n_1(R)}}{u^{\omega_0(1)(R)}}\right]
 =\alpha_1(R),
 $$
 with the notation of (\ref{E2}).
 We have that $R_1/m_{R_1}=R/m_R[\sigma(R)]$ and $S_1/m_{S_1}=S/m_S[\sigma(S)]$ (with notation of (\ref{eqN9})).
 $$
 \alpha_1(R)=\left[\frac{v^{\overline n_1(R)}}{u^{\omega_0(1)(R)}}\right]=\sigma(R)
 $$
 and
 $$
 \begin{array}{lll}
 \alpha_1(S)&=&\left[\frac{y^{\overline n_1(R)}}{u^{\omega_0(1)(R)}}\right]
 =\left[\frac{y^{\overline n_1(R)}}{\gamma_0^{\omega_0(1)(R)}x^{t\omega_0(1)(R)}}\right]\\
 &=& [\gamma_0]^{-\omega_0(1)(R)}
 \left[\frac{y^{\overline n_1(S)}}{x^{\omega_0(1)(S)}}\right]
 =[\gamma_0]^{-\omega_0(1)(R)}\sigma(S).
 \end{array}
 $$
 Thus
$R_1/m_{R_1}=R/m_R(\alpha_1(R))$ and $S_1/m_{S_1}=S/m_S(\alpha_1(R))$.

By (\ref{eq9}), we have that
$$
[S_1/m_{S_1}:S/m_S]=[R_1/m_{R_1}:R/m_R]
$$
and thus
$$
d_1(S)=[S/m_S(\alpha_1(R)):S/m_S]=[R/m_R(\alpha_1(R)):R/m_R]=d_1(R),
$$
 and the minimal polynomial $f_1^S(z)$ of $\alpha_1(S)$ over $S/m_S$ is the minimal polynomial $f_1^R(z)$ of $\alpha_1(R)$ over $R/m_R$. Thus
 $$
 x,y=P_1(R), P_2(R)
 $$
 are the first terms of a generating sequence in $S$, obtained by the (modified) algorithm of Theorem 4.2 \cite{CV}.

\begin{Proposition}\label{Prop30} Suppose that  $i\ge 2$ and
$$
P_0(S)=x, P_1(S)=y, P_2(S)=P_2(R),\ldots, P_i(S)=P_i(R)
$$
are the first $i+1$ terms of a generating sequence in $S$ produced by the modified algorithm of Theorem 4.2 \cite{CV}.
Then
$$
P_0(S)=x, P_1(S)=y, P_2(S)=P_2(R),\ldots, P_i(S)=P_i(R), P_{i+1}(S)=P_{i+1}(R)
$$
are the first $i+2$ terms of a generating sequence in $S$ produced by the modified algorithm of Theorem 4.2 \cite{CV}.
\end{Proposition}

\begin{proof} With the assumption,  we have that for $j\le i-1$,
$$
\overline n_j(S)=\overline n_j(R), \alpha_j(S)=\alpha_j(R),
$$
$$
\begin{array}{lll}
d_j(S)&=&[S/m_S(\alpha_1(S),\ldots,\alpha_j(S)):S/m_S(\alpha_1(S),\ldots, \alpha_{j-1}(S))]\\
&=&[R/m_R(\alpha_1(R),\ldots,\alpha_j(R)):R/m_R(\alpha_1(R),\ldots, \alpha_{j-1}(R))]=d_j(R)
\end{array}
$$
and the minimal polynomial $f_j^S(z)$ of $\alpha_j(S)=\alpha_j(R)$
over $S/m_S(\alpha_1(S),\ldots, \alpha_{j-1}(S))$
is the minimal polynomial $f_j^R(z)$ of $\alpha_j(R)$ over $R/m_R(\alpha_1(R),\ldots, \alpha_{j-1}(R))$.

Theorem 7.1 \cite{CV} produces a  generating sequence $Q_0(R_1)=u_1, Q_1(R_1),Q_2(R_1),\ldots$ in $R_1$
from $P_0(R), P_1(R),P_2(R),\ldots$. The generating  sequence $Q_0(R_1)=u_1, Q_1(R_1),Q_2(R_1),\ldots$ in $R_1$
  can be produced by the algorithm of Theorem 4.2 from the regular system of parameters $u_1, Q_1(R_1)$ in $R_1$ (as shown in Theorem 7.1 \cite{CV} and recalled in (\ref{eq10}) and (\ref{eqN11})).

Since $R\rightarrow S$ is stable, we have that
\begin{equation}\label{eq21}
u_1=\gamma_1x_1^t
\end{equation}
for some unit $\gamma_1\in S_1$, and recalling (\ref{eq12}), we have that
\begin{equation}\label{eq20}
\omega_0(1)(S)=t\omega_0(1)(R).
\end{equation}

By the induction hypothesis applied to the stable map $R_1\rightarrow S_1$,  we have that
$$
x_1, Q_1(R_1),\ldots, Q_i(R_1)
$$
are the first $i+1$ terms of a generating sequence in $S_1$, produced by the modified algorithm of Theorem 4.2 \cite{CV} in $S_1$.

For $j\ge 1$, let
$$
\overline n_j(R_1), U_j(R_1), \alpha_j(R_1), d_j(R_1), f_j^{R_1}(z)
$$
be the calculations of the algorithm of Theorem 4.2 \cite{CV} in $R_1$, obtained in the construction of the generating sequence
$Q_0(R_1)=u_1, Q_1(R_1),Q_2(R_1),\ldots$.

For $j\le i$, let
$$
\overline n_j(S_1), U_j(S_1), \alpha_j(S_1), d_j(S_1), f_j^{S_1}(z)
$$
be the calculations of the modified algorithm of Theorem 4.2 \cite{CV} in $S_1$, obtained in the construction of the first $i+1$ terms of the generating sequence $x_1, Q_1(R_1),Q_2(R_1),\ldots, Q_i(R_1)$ in $S_1$. We have that for $j\le i-1$,
\begin{equation}\label{eq22}
\begin{array}{l}
\overline n_j(S_1)=\overline n_j(R_1), U_j(S_1)=U_j(R_1), \alpha_j(S_1)=\alpha_j(R_1),\\
d_j(S_1)=d_j(R_1), f_j^{S_1}(u)=f_j^{R_1}(u).
\end{array}
\end{equation}

Since
$$
Q_0(S_1)^{\omega_0(1)(S)}=Q_0(R_1)^{\omega_0(1)(R)}\gamma_1^{-\omega_0(1)(R)},
$$
we have from (\ref{eqN11}) that for $j\le i-1$,
$$
\begin{array}{lll}
Q_j(S_1)&=&\gamma_1^{\omega_0(1)(R)n_1(S)\cdots n_j(S)}\frac{P_{j+1}(S)}{Q_0(R_1)^{\omega_0(1)(R)n_1(S)\cdots n_j(S)}}\\
&=&\gamma_1^{\omega_0(1)(R)n_1(S)\cdots n_j(S)} Q_j(R_1).
\end{array}
$$

For $j\le i-1$, we have by (\ref{eqN11}) and (\ref{eq22}), and then by (\ref{eq21}) and (\ref{eq20}), that
$$
\begin{array}{lll}
\nu^*(Q_j(R_1))&=& \nu^*(P_{j+1}(R))-\omega_0(1)(R)n_1(R)\cdots n_j(R)\nu^*(u_1)\\
&=& \nu^*(P_{j+1}(R))-n_1(R)\cdots n_j(R)\omega_0(1)(S)\nu^*(x_1).
\end{array}
$$
Thus
$$
\begin{array}{lll}
G(\nu^*(x_1),\nu(Q_1(R)),\ldots, \nu(Q_j(R)))&=& G(\nu^*(x_1),\nu(P_2(R)),\ldots, \nu(P_{j+1}(R)))\\
&=& G(\nu^*(x), \nu^*(y),\nu(P_2(R)),\ldots, \nu(P_{j+1}(R)))\\
&=& G(\nu^*(x),\nu(P_1(R)),\ldots,\nu(P_{j+1}(R)))
\end{array}
$$
since $G(\nu^*(x_1))=G(\nu^*(x),\nu^*(y))$, as calculated before (55)  in the proof of Theorem 7.1 \cite{CV}.
Thus $\overline n_{i-1}(S_1)=\overline n_i(S)$. We have that $\overline n_{i-1}(S_1)=\overline n_{i-1}(R_1)$ by (\ref{eq22}),
and $\overline n_{i-1}(R_1)=\overline n_i(R)$ by (55) and (54)  in the proof of Theorem 7.1 \cite{CV}.
Thus
$$
\overline n_i(S)=\overline n_i(R).
$$
In applying the modified algorithm of Theorem 4.2 \cite{CV} to extend $x,P_1(R),\ldots, P_i(R)$ to a generating sequence in $S$,
we can thus take $U_i(S)=U_i(R)$, and then
$$
\alpha_i(S)=\left[\frac{P_i(S)^{\overline n_i(S)}}{U_i(S)}\right] = \left[\frac{P_i(R)^{\overline n_i(R)}}{U_i(R)}\right]
=\alpha_i(R).
$$
We have from (\ref{E4}) that
$$
y_1=\frac{y^{\overline n_1(S)}}{x^{\omega_0(1)(S)}}=\gamma_0^{\omega_0(1)(R)}\frac{v^{\overline n_1(R)}}{u^{\omega_0(1)(R)}}
=\gamma_0^{\omega_0(1)(R)}v_1.
$$
Thus
$$
\sigma(S_1)=[y_1]=[\gamma_0]^{\omega_0(1)(R)}[v_1]=[\gamma_0]^{\omega_0(1)(R)}\alpha_1(R)
$$
in $V_{\nu^*}/m_{\nu^*}$, and
$$
S_1/m_{S_1}=S/m_S[\alpha_1(R)]\mbox{ and } R_1/m_{R_1}=R/m_R[\alpha_1(R)].
$$
For $1\le j\le i-1$, by (60) of \cite{CV}, we have that
$$
\alpha_j(S_1)=\alpha_j(R_1)=\hat{\alpha}_j(R_1)=\alpha_{j+1}(R)\alpha_1(R)^{a(R)\omega_0(i+1)(R)+b(R)\omega_1(i+1)(R)}.
$$
Thus
$$
\begin{array}{lll}
d_{i-1}(S_1)&=&[S_1/m_{S_1}(\alpha_1(S_1),\ldots, \alpha_{i-1}(S_1)):S_1/m_{S_1}(\alpha_1(S_1),\ldots, \alpha_{i-2}(S_1))]\\
&=& [S/m_S(\alpha_1(R),\alpha_2(R),\ldots, \alpha_{i}(R)):S/m_{S}(\alpha_1(R),\ldots, \alpha_{i-1}(R))]=d_i(S)
\end{array}
$$
and
$$
\begin{array}{lll}
d_{i-1}(R_1)&=&[R_1/m_{R_1}(\alpha_1(R_1),\ldots, \alpha_{i-1}(R_1)):R_1/m_{R_1}(\alpha_1(R_1),\ldots, \alpha_{i-2}(R_1))]\\
&=& [R/m_R(\alpha_1(R),\alpha_2(R),\ldots, \alpha_{i}(R)):R/m_{R}(\alpha_1(R),\ldots, \alpha_{i-1}(R))]=d_i(R).
\end{array}
$$
We thus have that $d_i(S)=d_i(R)$ since $d_{i-1}(S_1)=d_{i-1}(R_1)$ by (\ref{eq22}). Thus the minimal polynomial $f_i^S(z)$ of $\alpha_i(S)=\alpha_i(R)$ over $S/m_S(\alpha_1(R),\ldots, \alpha_{i-1}(R))$ is the minimal polynomial $f_i^R(z)$ of $\alpha_i(R)$ over $R/m_R(\alpha_1(R),\ldots, \alpha_{i-1}(R))$. Thus we can take $P_{i+1}(S)=P_{i+1}(R)$ in the modified algorithm of Theorem 4.2
\cite{CV}.

\end{proof}

We obtain the following theorem (Theorem \ref{Theorem2} from the Introduction of this paper).  Theorem \ref{Theorem1} is proven by Ghezzi, H\`a and Kashcheyeva in \cite{GHK} when $k$ is algebraically closed of characteristic zero.

\begin{Theorem}\label{Theorem1} Suppose that $k$ is a field of characteristic zero, $\nu^*$ is a rational 0-dimensional valuation, $n=2$  and $R\rightarrow S$ is stable.
Then
$$
{\rm gr}_{\nu^*}(S)\cong \left({\rm gr}_{\nu}(R)\otimes_{R/m_R}S/m_S\right)[Z]/(Z^e-[\gamma_0]^{-1}[u]),
$$
and the degree of the extension  of quotient fields of  ${\mbox gr}_{\nu}(R)\rightarrow \mbox{gr}_{\nu^*}(S)$ is $ef$.
\end{Theorem}

\begin{proof} We have an inclusion of graded algebras ${\rm gr}_{\nu}(R)\rightarrow {\rm gr}_{\nu^*}(S)$.
The classes $[P_i(R)]$ for $i\ge 0$ generate ${\rm gr}_{\nu}(R)$ as a ${\rm gr}_{\nu}(R)_0=R/m_R$-algebra and the classes
$[P_0(S)]$ and $[P_i(R)]$ for $i\ge 1$ generate ${\rm gr}_{\nu^*}(S)$ as a ${\rm gr}_{\nu}(S)_0=S/m_S$-algebra by Theorem 4.11 \cite{CV} and Proposition \ref{Prop30}. We have the relation
\begin{equation}\label{eq31}
[P_0(S)]^t[\gamma_0]=[P_0(R)]
\end{equation}
in ${\rm gr}_{\nu^*}(S)$. Further,
\begin{equation}\label{eq32}
n_i(R)=n_i(S)\mbox{ for }i\ge 1
\end{equation}
by Proposition \ref{Prop30}.

Since ${\rm gr}_{\nu}(R)\otimes_{R/m_R}S/m_S\rightarrow {\rm gr}_{\nu^*}(S)$
is homogeneous, to verify that it is 1-1, it suffices to show that the homomorphism of $S/m_S$-vector spaces
\begin{equation}\label{eq33}
{\rm gr}_{\nu}(R)_{\lambda}\otimes_{R/m_R}S/m_S\rightarrow {\rm gr}_{\nu^*}(S)_{\lambda}
\end{equation}
is 1-1 for all $\lambda\in S^{R}(\nu)$. By 2) of Theorem 4.2 \cite{CV}, the set of all monomials
\begin{equation}\label{eq34}
[P_0(R)]^{i_0}[P_1(R)]^{i_1}\cdots[P_r(R)]^{i_r}
\end{equation}
such that $r\in \NN$, $i_k\in \NN$, $0\le i_k<n_k(R)$ for $1\le k\le r$ and
$$
i_0\nu(P_0(R))+\cdots+i_r\nu(P_r(R))=\lambda
$$
is an $R/m_R$-basis of ${\rm gr}_{\nu}(R)_{\lambda}$, and the set of all
\begin{equation}\label{eq35}
[P_0(S)]^{j_0}[P_1(S)]^{j_1}\cdots[P_s(S)]^{j_s}
\end{equation}
such that $s\in \NN$, $j_k\in \NN$, $0\le j_k<n_k(S)$ for $1\le k\le s$ and
$$
j_0\nu^*(P_0(S))+\cdots+j_s\nu^*(P_s(S))=\lambda
$$
is an $S/m_S$-basis of ${\rm gr}_{\nu^*}(S)_{\lambda}$.

By (\ref{eq31}), (\ref{eq32}), (\ref{eq34}) and (\ref{eq35}), we have that (\ref{eq33}) is 1-1, so
$$
{\rm gr}_{\nu}(R)\otimes_{R/m_R}S/m_S\rightarrow {\rm gr}_{\nu^*}(S)
$$
is 1-1.

We have established that $[P_0(S)]$ generates ${\rm gr}_{\nu^*}(S)$ as a ${\rm gr}_{\nu}(R)\otimes_{R/m_R}S/m_S$-algebra and that the relation (\ref{eq31}) holds. To establish that the conclusions of the theorem hold, we must show that if there is a relation
\begin{equation}\label{eq36}
h_0+[P_0(S)]h_1+\cdots+[P_0(S)]^{t-1}h_{t-1}=0
\end{equation}
in ${\rm gr}_{\nu^*}(S)$, with $h_i\in {\rm gr}_{\nu}(R)\otimes_{R/m_R}S/m_S$, then $h_0=h_1=\cdots=h_{t-1}=0$. We may
assume that each $[P_0(S)]^jh_j$ is homogeneous of the same degree $\lambda$. Since $R\rightarrow S$ is stable, we have that $t=[\Gamma_{\nu^*}:\Gamma_{\nu}]$ and
$i\nu(P_0(S))\not\in \Gamma_{\nu}$ for $1\le i\le t-1$. Thus there can be at most one nonzero expression in (\ref{eq36}), so all terms are zero.

\end{proof}


\begin{thebibliography}{1000000000}
\bibitem{Ab1} S. Abhyankar, On the Valuations centered in a Local Domain, Amer. J. Math., Vol 78., 1956.
\bibitem{Ab} S. Abhyankar, Ramification theoretic method in Algebraic Geometry, Princeton University Press, Princeton, New Jersey, 1959.
\bibitem{C} S.D. Cutkosky, Local factorization and monomialization of morphisms, Ast\'erisque 260, 1999.
\bibitem{CP} S.D. Cutkosky and O. Piltant, Ramification of valuations, Advances in Math. 183 (2004), 1-79.
\bibitem{CT} S.D. Cutkosky and B. Teissier, Semigroups of valuations on local rings,
Mich. Math. J. 57 (2008), 173 - 193.
\bibitem{CV} S.D. Cutkosky and Pham An Vinh, Valuation semigroups of two dimensional local rings, to appear in the Proceedings  of the London Math. Soc.
\bibitem{GHK} L. Ghezzi, Huy T\`ai H\`a and O. Kashcheyeva, Toroidalization of generating sequences in dimension two function fields, J. Algebra 301 (2006) 838-866.
\bibitem{GH} L. Ghezzi and O. Kashcheyeva, Toroidalization of generating sequences in dimension two function fields of positive characteristic, J. Pure and Applied Algebra 209 (2007), 631 - 649.
\bibitem{KK} H. Knaf and F.-V. Kuhlmann, Abhyankar places admit local uniformization in any characteristic, Ann. Scient. Ec. Norm. Sup 30 (2005), 833 - 846.
\bibitem{T1} B. Teissier, Valuations, Deformations and Toric geometry, Proceedings of the Saskatoon Conference and Workshop on Valuation Theory, Volume II, F.-V. Kuhlmann, S. Kuhlmann, M. Marshall (Eds.), Fields Inst. Comm. 33, AMS 2003.
\bibitem{T} B. Teissier, Overweight deformations of affine toric varieties and local uniformization, preprint.
\bibitem{ZS} O. Zariski and P. Samuel, Commutative Algebra, Volume 2, Van Nostrand, Princeton, New Jersey, 1960.
\end{thebibliography}
\end{document}